\documentclass{article}
\usepackage[T1]{fontenc}
\usepackage{amssymb,amsmath,amsfonts,amsthm}
\usepackage{mathtools} 
\usepackage[arrow,matrix]{xy}
\usepackage[hidelinks]{hyperref}
\usepackage{paralist} 
\usepackage{graphicx,comment,color,xifthen}
\usepackage{rotating,stmaryrd} 
\usepackage[footnotesep=2\baselineskip]{geometry}
\def\calA{\mathcal{A}}  \def\calC{\mathcal{C}} 
  \def\calG{\mathcal{G}}

 \def\frakF{\mathfrak{F}}

 \def\bfB{\mathbf{B}}

 \def\rmB{\mathrm{B}}

 \def\dbN{\mathbb{N}}

\theoremstyle{plain}
\newtheorem{theorem}{Theorem}[section]
\newtheorem{conjecture}[theorem]{\textcolor{red}{Conjecture}}

\newtheorem{lemma}[theorem]{Lemma}
\newtheorem{proposition}[theorem]{Proposition}
\theoremstyle{definition}
\newtheorem{definition}[theorem]{Definition}

\newtheorem*{notation}{Notation}
\newtheorem*{remark}{Remark}
\newcommand{\Set}[1][]{\ifthenelse{\isempty{#1}}{\mathrm{Set}}{\mathrm{Set}_{\tiny #1}}}
\newcommand{\DDelta}{{\raisebox{-1pt}{\begin{sideways}\ensuremath{\trianglerighteqslant}\end{sideways}}}}

\begin{document}

\title{Cats}
\author{\textsc{Daniel Gerigk}\footnote{\,Universit\"at Bonn, Germany. \quad \href{mailto:danger@uni-bonn.de}{danger@uni-bonn.de}}}
\date{March 7, 2014}
\maketitle

\begin{center} \textsl{To the memory of Christoph S.} \end{center}

\begin{abstract} \noindent
A generalization of the notion of an $\infty$-category is presented, allowing for ($\infty$-)cat(egorie)s that may have non-invertible higher morphisms.\\
The first step is to find a suitable category $\DDelta$ of (generalized) simplices. In fact, the category $\DDelta$ which we will employ has already been introduced a long time ago. Consider $\Set[\DDelta]$. Every simplex $A \in \DDelta$ has \emph{(inner) faces}, corresponding \emph{(inner) horns}, and a \emph{spine}. We call an object $X \in \Set[\DDelta]$ a \emph{cat} if every inner horn in $X$ can be filled. We conjecture that every spine is \emph{inner anodyne}, and that the Cisinski model structure generated by the set of spines is equal to the Cisinski model structure generated by the set of inner horns. It is conjectured that the fibrant objects of this model structure are precisely the cats.
\end{abstract}

\section{Preface}

The notion of an \emph{$\infty$-category} was first defined by Boardman \& Vogt \cite{boardmanvogt} (who call them \emph{weak Kan complexes}) and was further developed most notably by Joyal \cite{joyal.qcakc,joyal.qcvsss,joyal.noqc,joyal.ttoqcaia} (who calls them \emph{quasi-categories}) and later by Lurie \cite{lurie.htt,lurie.ha}.

\begin{notation}
Denote by $\Delta$ the category of (classical) simplices. Define $\dbN:=\{0,1,2,\ldots\}$ and $\dbN_+:=\dbN \setminus \{0\}$.
\end{notation}

\section{The category of (generalized) simplices}

\begin{definition}
Define $\calA:=\prod'_{n \in \dbN_+}\Delta \subset \prod_{n \in \dbN_+}\Delta$ to be the full subcategory generated by the objects which have only finitely many components not equal to $\Delta^0$. For a morphism $f \colon A \to B$ in $\calA$, define $\deg f$ to be the smallest $k \in \dbN_+$ such that $f_k$ is constant, i.\,e. factors through $\Delta^ 0$. For morphisms $f,g \colon A \to B$ in $\calA$, define $f \sim g$ if and only if $\deg f = \deg g =: d$ and $f_1=g_1,\ldots,f_d=g_d$. This defines an equivalence relation on the set of morphisms in $\calA$ which is compatible with composition of morphisms. For $A \in \calA$, define $\dim A$ to be the smallest $d \in \dbN$ such that $A_{d+1} = \Delta^0$. Define $\DDelta \subset \calA/{\sim}$ to be the full subcategory generated by the objects $A=(A_1,A_2,\ldots)$ having the property that $A_n = \Delta^0$ for all $n > \dim A$. 
\end{definition}

The category $\DDelta$ was introduced by Simpson \cite{simpson.acmsfncihnsagsvk} under the notation $\Theta$. 

\begin{notation}
For $n \in \dbN$ and $a_1,\ldots,a_n \in \dbN_+$, define \[\DDelta^{a_1,\ldots,a_n} := (\Delta^{a_1},\ldots,\Delta^{a_n},\Delta^0,\Delta^0,\ldots) \in \DDelta.\]
In particular, $\DDelta^\emptyset = (\Delta^0,\Delta^0,\ldots)$.
\end{notation}

\begin{proposition}
$\DDelta$ is a skeleton of $\calA/{\sim}$, and the objects in $\DDelta$ have no non-identity automorphisms. Two objects $A,B \in \calA$ become isomorphic in $\calA/{\sim}$ if and only if $\dim A =\dim B =: d$ and $A_1 = B_1, \ldots, A_d = B_d$.
\end{proposition}

\begin{theorem}[Ara \& Maltsiniotis]
The category $\DDelta$ is a strict test category in the sense of Grothendieck \cite{grothendieck.ps}.
\end{theorem}

\section{Faces, horns and spines}

\begin{definition}
Let $A \in \DDelta$. The maximal proper subobjects $B \hookrightarrow A$ in $\DDelta$, and also the images of the corresponding monomorphisms $B \hookrightarrow A$ in $\Set[\DDelta]$, are called \emph{faces} of $A$.
\end{definition}

\begin{lemma}
Let $F \colon B \hookrightarrow A$ be a face in $\DDelta$, and define $d := \dim A$. Then $d-1 \leq \dim B \leq d \leq \deg F \leq d+1$, and the components $F_1,\ldots,F_d$ of $F$ are well-defined monomorphisms in $\Delta$. There is a unique $k \in \{1,\ldots,d\}$ such that the monomorphism $F_k \colon B_k \hookrightarrow A_k$ is a (classical) face in $\Delta$, and for $n \in \{1,\ldots,d\} \setminus \{k\}$ the monomorphism $F_n$ is an identity.
\end{lemma}

\begin{definition}
Let $A = \DDelta^{a_1,\ldots,a_d} \in \DDelta$, $k \in \{1,\ldots,d\}$ and $m \in \{0,\ldots,a_k\}$. Denote by $\delta_{k:m}^A \subset A$ the face whose $k$-th component is equal to the (classical) face $\delta_m^{A_k} \subset A_k = \Delta^{a_k}$. It is called an \emph{inner face} if the $k$-th component is a (classical) inner face.
\end{definition}

\begin{proposition}
The simplex $A = \DDelta^{a_1,\ldots,a_d}$ has precisely $\sum_{i=1}^d(a_i-1)$ many inner faces, and precisely $2 \,| \{ i \colon a_i \geq 2\} \cup \{d\}|$ many outer faces.
\end{proposition}

\begin{definition}
Let $A \in \DDelta$. A simplicial subset $\Lambda \subset A$ is called a \emph{horn} if there is a face $B \subset A$ such that $\Lambda$ is the union of all faces of $A$ except $B$. It is called an \emph{inner horn} if the missing face is inner.
\end{definition}

\begin{definition}
Let $A = \DDelta^{a_1,\ldots,a_d} \in \DDelta$, $k \in \{1,\ldots,d\}$ and $m \in \{0,\ldots,a_k\}$. Define $\Lambda^{k:m} \subset A$ to be the horn whose missing face is $\delta_{k:m}^A \subset A$.
\end{definition}

\begin{definition}
A morphism $f \colon X \to Y$ in $\Set[\DDelta]$ is called an \emph{inner fibration} if it has the right lifting property with respect to every inner horn.
\end{definition}

\begin{definition}
A morphism $f \colon X \to Y$ in $\Set[\DDelta]$ is called \emph{inner anodyne} if it has the left lifting property with respect to every inner fibration.
\end{definition}

\begin{definition}
Let $A \in \DDelta$. Define the \emph{boundary} $\partial A \subset A$ to be the union of all faces of $A$.
\end{definition}

\begin{proposition}
The class of monomorphisms in $\Set[\DDelta]$ is generated (as a saturated class) by the set of boundaries $\partial A \subset A$, with $A \in \DDelta$.
\end{proposition}

\begin{definition}
Let $A \in \DDelta$. Define the \emph{spine} $I(A) \subset A$ to contain a simplex $s \colon B \to A$ if and only if for all $k \in \{1,\ldots,\deg s\}$, the $k$-th component $s_k \colon B_k \to A_k$ is contained in the (classical) spine $I(A_k) \subset A_k$.
\end{definition}

\section{Cats and groupoids}

\begin{definition}
$\calC \in \Set[\DDelta]$ is called a \emph{cat} if for every inner horn $\Lambda \subset A$ the map $\calC(A) \to \calC(\Lambda)$ is surjective. $\calG \in \Set[\DDelta]$ is called a \emph{groupoid} if for every horn $\Lambda \subset A$ the map $\calG(A) \to \calG(\Lambda)$ is surjective.
\end{definition}

\begin{definition}
A cat $\calC$ is called \emph{strict} if for every inner horn $\Lambda \subset A$ the map $\calC(A) \to \calC(\Lambda)$ is bijective. A groupoid $\calG$ is called \emph{strict} if for every horn $\Lambda \subset A$ the map $\calG(A) \to \calG(\Lambda)$ is bijective.
\end{definition}

\begin{proposition}
A groupoid is strict if and only if it is strict when viewed as a cat.
\end{proposition}

\begin{definition}
Let $n \in \dbN_+$. A cat $\calC$ is called $n$-strict if for every inner horn $\Lambda \subset A$ with $\dim A \geq n$, the map $\calC(A) \to \calC(\Lambda)$ is bijective.
\end{definition}
\begin{remark}
A cat is strict if and only if it is 1-strict.
\end{remark}

\section{The model structure for cats}

\begin{definition}
The Cisinski model structure on $\Set[\DDelta]$ generated by the set of spines is called the \emph{model structure for cats}. The weak equivalences of this model structure are called \emph{weak cat equivalences}.
\end{definition}

\begin{definition}
Let $A \in \DDelta$. Denote by $\frakF_A$ the set of faces of $A$, by $\frakF_A^i$ the set of inner faces of $A$, and by $\frakF_A^o$ the set of outer faces of $A$.
\end{definition}

\begin{lemma}
\begin{compactenum}[(i)]
\item $I(\DDelta^{1,\ldots,1}) = \DDelta^{1,\ldots,1}$.
\item If $A \in \DDelta$ is not of the form $A=\DDelta^{1,\ldots,1}$, then $I(A) \subset \bigcup \frakF_A^o$.
\end{compactenum}
\end{lemma}

\begin{lemma}[cf. Proposition 2.12 in Joyal \cite{joyal.ttoqcaia}]
For every $A \in \DDelta$ and every subset $\Gamma \subsetneq \frakF_A$ containing all outer faces of $A$, the inclusion $\bigcup \Gamma \subset A$ is inner anodyne.
\end{lemma}
\begin{proof}
We proceed by induction on $A$: let $A \in \DDelta$, and assume that for every face $B$ of $A$ and every subset $\Gamma \subsetneq \frakF_B$ containing all outer faces of $B$, the inclusion $\bigcup \Gamma \subset B$ is inner anodyne. Let $\Gamma \subsetneq \frakF_A$ contain all outer faces of $A$. If $\Gamma$ contains all faces of $A$ except one inner face, then the inclusion $\bigcup \Gamma \subset A$ is an inner horn of $A$, so is inner anodyne. Assume that $\Gamma$ misses at least two inner faces of $A$. It suffices to show that for $B \in \frakF_A \setminus \Gamma$, the inclusion $\beta \colon \bigcup \Gamma \subset \bigcup \Gamma \cup B$ is inner anodyne.

We have a pushout diagram
$$\begin{xy}\xymatrix{
\bigcup_{F \in \Gamma} (F \cap B) \ar@{^{(}->}[r] \ar@{^{(}->}[d]_\alpha & \bigcup \Gamma \ar@{^{(}->}[d]^\beta\\
B \ar@{^{(}->}[r] & \bigcup \Gamma \cup B
}\end{xy}$$

The subset $\{F \cap B \colon F \in \Gamma\} \subset \frakF_B$ is proper and contains all outer faces of $B$. Hence, by the induction hypothesis on $B$, it follows that $\alpha$ is inner anodyne. Because the class of inner anodyne maps is stable under pushouts, we conclude that $\beta$ is inner anodyne.
\end{proof}

\begin{conjecture}[cf. Proposition 2.13 in Joyal \cite{joyal.ttoqcaia}]
For every $A \in \DDelta$ the inclusion $I(A) \subset I(A) \cup \bigcup \frakF_A^o$ is inner anodyne.
\end{conjecture}

From now on, assume the previous conjecture to be proven.

\begin{proposition}[cf. Proposition 2.13 in Joyal \cite{joyal.ttoqcaia}]
Every spine $I(A) \subset A$, $A \in \DDelta$, is inner anodyne.
\end{proposition}

\begin{conjecture}[cf. Lemma 3.5 in Joyal \cite{joyal.qcvsss}]
\label{spineconjecture2}
For every $A \in \DDelta$, the inclusion $I(A) \subset I(A) \cup \bigcup \frakF_A^o$ is a weak cat equivalence.
\end{conjecture}

From now on, assume the previous conjecture to be proven.

\begin{lemma}[cf. Lemma 3.5 in Joyal \cite{joyal.qcvsss}]
For every $A \in \DDelta$ and every subset $\Gamma \subsetneq \frakF_A$ containing all outer faces of $A$, the inclusion $\bigcup \frakF_A^o \subset \bigcup \Gamma$ is a weak cat equivalence.
\end{lemma}
\begin{proof}
We prove this by induction on $A$: let $A \in \DDelta$, and assume that for every face $B$ of $A$ and every subset $\Gamma \subsetneq \frakF_B$ containing all outer faces of $B$, the inclusion $\bigcup \frakF_B^o \subset \bigcup \Gamma$ is a weak cat equivalence.

Let $\Gamma \subsetneq \frakF_A$ contain all outer faces of $A$. Assume $B \in \frakF_A \setminus \Gamma$ such that $\Gamma \cup \{B\} \subsetneq \frakF_A$. We show that the inclusion $\beta \colon \bigcup \Gamma \subset \bigcup \Gamma \cup B$ is a weak cat equivalence.

We have a pushout diagram
$$\begin{xy}\xymatrix{
\bigcup_{F \in \Gamma} (F \cap B) \ar@{^{(}->}[r] \ar@{^{(}->}[d]_\alpha & \bigcup \Gamma \ar@{^{(}->}[d]^\beta\\
B \ar@{^{(}->}[r] & \bigcup \Gamma \cup B
}\end{xy}$$

Consider $I(B) \hookrightarrow \bigcup_{F \in \Gamma} (F \cap B) \xhookrightarrow{\alpha} B$. By the induction hypothesis on $B$ and conjecture \ref{spineconjecture2}, the first map is a weak cat equivalence. The composition $I(B) \hookrightarrow B$ is a weak cat equivalence by definition. It follows that $\alpha$ is a weak cat equivalence. Because the class of monomorphisms that are weak cat equivalences is stable under pushouts, we conclude that $\beta$ is a weak cat equivalence.
\end{proof}

\begin{proposition}
Every inner anodyne map is a weak cat equivalence.
\end{proposition}

\begin{proposition}
The model structure for cats is equal to the Cisinski model structure generated by the set of inner horns. In particular, if an object $X \in \Set[\DDelta]$ is fibrant with respect to the model structure for cats, then it is a cat.
\end{proposition}

\begin{conjecture}
The fibrant objects of the model structure for cats are precisely the cats.
\end{conjecture}

\section{$\mathbf{n}$-cats}

In this section, let $n \in \dbN$.

\begin{definition}
Define $\DDelta_n \subset \DDelta$ to be the full subcategory generated by the objects $A \in \DDelta$ with $\dim A \leq n$.
\end{definition}

\begin{proposition}
The sequence of canonical projection functors $\DDelta_0 \leftarrow \DDelta_1 \leftarrow \DDelta_2 \leftarrow \ldots \leftarrow \DDelta$ induces a sequence of (full and faithful) embeddings $\Set[\DDelta_0] \hookrightarrow \Set[\DDelta_1] \hookrightarrow \Set[\DDelta_2] \hookrightarrow \ldots \hookrightarrow \Set[\DDelta]$.
\end{proposition}

\begin{definition}
$\calC \in \Set[\DDelta_n]$ is called an \emph{$n$-cat} if for every inner horn $\Lambda \subset A$ with $A \in \DDelta_n$, the map $\calC(A) \to \calC(\Lambda)$ is surjective.
\end{definition}

\begin{remark}
For a cat $\calC \in \Set[\DDelta]$, the restriction $\calC_{|\tiny \DDelta_n} \in \Set[\DDelta_n]$ is an $n$-cat.
\end{remark}

\begin{remark}
Beware that an $n$-cat $\calC \in \Set[\DDelta_n]$ is in general not a cat when viewed as an object of $\Set[\DDelta]$ via the embedding $\Set[\DDelta_n] \hookrightarrow \Set[\DDelta]$.
\end{remark}

\begin{definition}
A cat $\calC \in \Set[\DDelta]$ is called an \emph{$n$-cat} if $\calC$ is weakly cat equivalent to an object in $\Set[\DDelta_n]$ (viewed as an object in $\Set[\DDelta]$).
\end{definition}

From now on, assume $n \geq 1$.

\begin{definition}
The Cisinski model structure on $\Set[\DDelta_n]$ generated by the set of spines which are contained in $\Set[\DDelta_n]$ is called the \emph{model structure for $n$-cats}.
\end{definition}

\begin{proposition}
The model structure for $n$-cats is equal to the Cisinski model structure generated by the set of inner horns which are contained in $\Set[\DDelta_n]$. In particular, if an object $X \in \Set[\DDelta_n]$ is fibrant with respect to the model structure for $n$-cats, then it is an $n$-cat.
\end{proposition}

\begin{conjecture}
The fibrant objects of the model structure for $n$-cats are precisely the $n$-cats.
\end{conjecture}

\section{The model structure for groupoids}

\begin{definition}
The Cisinski model structure on $\Set[\DDelta]$ generated by the set of maps $A \to 1$, with $A \in \DDelta$, is called the \emph{model structure for groupoids}. The weak equivalences of this model structure are called \emph{weak groupoid equivalences} or \emph{weak homotopy equivalences}.
\end{definition}

\begin{conjecture}
The model structure for groupoids is equal to the Cisinski model structure generated by the set of horns. In particular, if an object $X \in \Set[\DDelta]$ is fibrant with respect to the model structure for groupoids, then it is a groupoid.
\end{conjecture}

\begin{conjecture}
The fibrant objects of the model structure for groupoids are precisely the groupoids.
\end{conjecture}

\section{$\mathbf{H^2(G;A)}$}

\begin{definition}
For a group $G$, define $\bfB^1 G$ to be the strict 1-groupoid which has a single object, whose 1-morphisms are in bijective correspondence with the elements of $G$, and whose composition of 1-morphisms corresponds to multiplication in $G$.
\end{definition}

\begin{definition}
For an abelian group $A$, define $\rmB^2 A$ to be the strict 2-cat which has a single object, a single 1-morphism, whose 2-morphisms are in bijective correspondence with the elements of $A$, and whose (vertical and horizontal) composition of 2-morphisms corresponds to addition in $A$.
\end{definition}

\begin{remark}
In general, $\rmB^2 A$ is not a groupoid. In fact, the map $(\rmB^2 A)(\DDelta^{2,1}) \to (\rmB^2 A)(\Lambda^{2:0})$ induced by the outer horn $\Lambda^{2:0} \subset \DDelta^{2,1}$ is not surjective if $A \neq 0$.
\end{remark}

\begin{definition}
We can define a groupoid $\bfB^2 A$ having no non-degenerate simplices of dimension $>2$ together with a weak homotopy equivalence $\rmB^2 A \hookrightarrow \bfB^2 A$. (Hint: $\bfB^2 A(\DDelta^2) := A$.)
\end{definition}

\begin{proposition}
There is a canonical bijective correspondence between the set of maps $\bfB^1 G \to \bfB^2 A$ and the set of 2-cocycles $G \times G \to A$.
\end{proposition}
\begin{conjecture}
The bijective correspondence of the previous proposition descends to a bijective correspondence between the set of homotopy equivalence classes of maps $\bfB^1 G \to \bfB^2 A$ and the set $H^2(G;A)$ of 2-cocycles modulo 2-coboundaries.
\end{conjecture}

\section{Personal note}

I would like to write a thesis about this topic, but haven't been able to find an advisor yet. If there is anyone willing to work with me on this, please let me know.

\bigskip
\begin{center}\includegraphics[width=4cm]{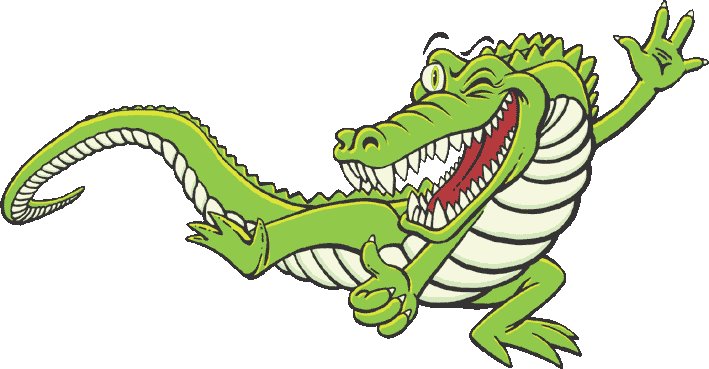}\end{center}
\end{document}